\immediate\write18{bibtex \jobname}

\documentclass[a4paper,USenglish,cleveref, autoref]{lipics-v2019}

\hideLIPIcs

\title{Parametrized Complexity of Expansion Height}

\author{Ulrich Bauer}{Department of Mathematics, Technical University of Munich (TUM)\\{Boltzmannstr. 3, 85748 Garching b. M\"unchen, Germany}}{ulrich.bauer@tum.de}{https://orcid.org/0000-0002-9683-0724}{}
\author{Abhishek Rathod}{Department of Mathematics, Technical University of Munich (TUM)\\{Boltzmannstr. 3, 85748 Garching b. M\"unchen, Germany}}{abhishek.rathod@tum.de}{https://orcid.org/0000-0003-2533-3699}{}
\author{Jonathan Spreer}{School of Mathematics and Statistics, The University of Sydney\\{NSW 2006 Australia}}{jonathan.spreer@sydney.edu.au}{https://orcid.org/0000-0001-6865-9483}{The author is partially supported by grant EVF-2015-230 of the Einstein Foundation Berlin. The authors acknowledges support by The University of Sydney, where parts of this work were finished.}
\authorrunning{U. Bauer, A. Rathod and J. Spreer}

\Copyright{Ulrich Bauer, Abhishek Rathod and Jonathan Spreer}%


\funding{This research has been supported by the DFG Collaborative Research Center SFB/TRR 109 ``Discretization in Geometry and Dynamics''}%
\acknowledgements{We want to thank Jo\~ao Paix\~ao for very helpful discussions.}%

\nolinenumbers %

\EventEditors{Michael A. Bender, Ola Svensson, and Grzegorz Herman}
\EventNoEds{3}
\EventLongTitle{27th Annual European Symposium on Algorithms (ESA 2019)}
\EventShortTitle{ESA 2019}
\EventAcronym{ESA}
\EventYear{2019}
\EventDate{September 9--11, 2019}
\EventLocation{Munich/Garching, Germany}
\EventLogo{}
\SeriesVolume{144}
\ArticleNo{12}

\usepackage{xspace}

\theoremstyle{definition}

\usepackage{tikz}
\usetikzlibrary{shapes.geometric}
\usetikzlibrary{calc}
\usetikzlibrary{arrows}

\newcounter{problemcounter}
\newenvironment{problem}[4][\unskip]
{
	\medskip\noindent {\textbf{Problem
	\arabic{problemcounter}} (\textsc{#1}).\nopagebreak

	{\begin{tabular}{lp{0.75\textwidth}}	
		\textsc{Instance}: & #2 \\
		\textsc{Parameter}: & #3 \\
		\textsc{Question}: & #4 \\
	\end{tabular}}
	\stepcounter{problemcounter}}
	\medskip
}

\newenvironment{problemwp}[3][\unskip]
{
	\medskip\noindent {\textbf{Problem
	\arabic{problemcounter}} (\textsc{#1}).\nopagebreak

	{\begin{tabular}{lp{0.75\textwidth}}	
		\textsc{Instance}: & #2 \\
		\textsc{Question}: & #3 \\
	\end{tabular}}
	\stepcounter{problemcounter}}
	\medskip
}

\DeclareMathOperator{\se}{\diagup \! \! \! \! \: \! \: \searrow}

\newcommand{\bbm}{\mathbb{M}}
\newcommand{\bbn}{\mathbb{N}}

\newcommand{\SCC}{\mathcal{S}}

\newcommand{\KCC}{\mathcal{K}}

\newcommand{\SSS}{\mathscr{S}}

\newcommand{\complex}{K}
\newcommand{\altcomplex}{L}

\newcommand{\face}{\tau}
\newcommand{\smallface}{\sigma}

\newcommand{\horn}{\mathbf{H}}

\newcommand{\str}{\operatorname{star}}

\newcommand{\expansion}[2]{\nearrow^{#2}_{#1}}
\newcommand{\collapse}[2]{\searrow^{#2}_{#1}}

\newcommand{\gadget}{\mathbf{P}}

\newcommand{\dunce}{\mathbf{D}}

\newcommand{\mathcommand}[3]{\newcommand{#1}[1]{\ensuremath{#2{##1}#3}}}

\newcommand\restr[2]{{%
  \left.\kern-\nulldelimiterspace %
  #1 %
  \vphantom{\big|} %
  \right|_{#2} %
  }}

\mathcommand{\Mmatchset}{\mathcal{M}(}{)}
\mathcommand{\Mmatchsetr}{{\mathcal{M}}(}{)}

\newcommand{\paraexp}{\textsc{Ordered Erasibility Expansion Height}\xspace}
\newcommand{\paraexptwo}{\textsc{Erasibility Expansion Height}\xspace}

\newcommand{\pMAS}{\textsc{Axiom Set}\xspace}
\newcommand{\MAS}{\textsc{Axiom Set}\xspace}

\newcommand{\closed}{untouched\xspace}
\newcommand{\open}{touched\xspace}

\newcommand{\oMm}{\textsc{Optimal Morse Matching}\xspace}

\newcommand{\WP}{{\bf W{[P]}}\xspace}

\newcommand{\NP}{{\bf NP}\xspace}

\newcommand{\dshe}{\textsc{Erasibility 3-Expansion Height}\xspace}
\newcommand{\dshetwo}{\textsc{Ordered Erasibilty 3-Expansion Height}\xspace}

\newcommand{\alphastrings}{\Sigma^{\ast}}

\definecolor{lightergray}{rgb}{0.8,0.8,0.8}
\definecolor{verylightgray}{rgb}{0.95,0.95,0.95}
\definecolor{verylightblue}{rgb}{0.925,0.95,1.0}
\definecolor{lightred}{rgb}{1.0,0.85,0.85}

\usepackage[colorinlistoftodos]{todonotes}

\begin{document}

\maketitle

\begin{abstract}
Deciding whether two simplicial complexes are homotopy equivalent is a fundamental problem in topology, which is famously undecidable.
There exists a combinatorial refinement of this concept, called simple-homotopy equivalence: two simplicial complexes are of the same simple-homotopy type if they can be transformed into each other by a sequence of two basic homotopy equivalences, an elementary collapse and its inverse, an elementary expansion. 
In this article we consider the following related problem: given a $2$-dimensional simplicial complex, is there a simple-homotopy equivalence to a $1$-dimensional simplicial complex using at most~$p$~expansions?
We show that the problem,  which we call the \emph{erasability expansion height}, is \WP-complete in the natural parameter~$p$.

\end{abstract}

\section{Introduction}

Homotopy theory lies at the heart of algebraic topology. In an attempt to make the concept of homotopy equivalence more amenable to combinatorial methods, Whitehead developed what turned out to be a combinatorial refinement of the theory, called \emph{simple-homotopy theory}. Simple-homotopy theory considers sequences of elementary homotopy equivalences defined on simplicial complexes (or, more generally, CW complexes): an \emph{elementary collapse}, which takes a face of a complex contained only in a single proper coface and removes both faces, and its inverse operation, called an \emph{elementary expansion}. 
Two simplicial complexes are then said to be of the same simple-homotopy type if one can be transformed into the other by a sequence of elementary collapses and expansions. Complexes of the same simple-homotopy type are homotopy equivalent, but the converse is not always true \cite{MR0005352}, the obstruction being an element of the \emph{Whitehead group} of the fundamental group.
However, Whitehead proved that all homotopy-equivalent complexes with a trivial fundamental group are in fact of the same simple-homotopy type \cite{MR0035437}, and thus in this particular case the notions of simple-homotopy and homotopy coincide.  

A presentation of the fundamental group can be read off from a two-dimensional complex such that the presentation is balanced and describes the trivial group if and only if this complex is contractible~\cite{MR1279174}. Since the decidability of the triviality problem for balanced presentations is open~\cite{MR1921705}, the same is also true for the decidability of contractibility of 2-complexes.
Hence, the decidability of the existence of a simple-homotopy equivalence from a 2-complex to a point is also open.   
In contrast, the problem of deciding whether a given complex has trivial fundamental group is famously undecidable already for $2$-complexes through its connection to the word problem, see, for instance, \cite{MR0260851}. 
It follows that sequences of elementary collapses and expansions proving simple-homotopy equivalence between a $2$-complex and a point can be expected to be long, if not unbounded. Nonetheless, understanding these sequences offers a great reward: the statement that any contractible $2$-complex contracts to a point using only expansions up to dimension three is equivalent to a weaker variant \cite[p.~34--35]{MR2341532} of the Andrews--Curtis conjecture \cite{MR0173241,MR0380813}.

In this article, motivated by the aforementioned problems, we investigate the computational (parametrized) complexity of a number of variants of the problem of deciding contractibility.
More precisely, we focus on the problem of deciding whether a given $2$-complex admits a simple-homotopy to a $1$-complex using at most $p$ expansions, called \paraexptwo. In addition, we consider a variant, called \paraexp, which requires that all expansions come at the very beginning of the sequence. It is worth noting that \paraexptwo and \paraexp are equivalent for CW complexes for which one can readily swap the order of expansions and collapses~\cite[p.~34]{MR2341532}. However, for simplicial complexes, the ordered and unordered expansion heights may differ. 

In \Cref{sec:hardness}, we prove that \paraexptwo is \WP-hard, see \Cref{thm:hardness}. The proof uses a reduction from \pMAS. The same reduction also establishes \WP-hardness of \paraexp. Note that a reduction from \pMAS is also used by the third author and others in \cite{MR3472422} to establish \WP-hardness of a parametrized version of \oMm. However, unlike in~\cite{MR3472422}, the use of combinatorial and topological properties of the dunce hat is a key ingredient of the reduction used in this paper. In particular, there is only one gadget in the reduction -- a subdivision of the so-called \emph{modified dunce hat}~\cite{MR3909634}, see \Cref{fig:gadget}. In this sense the techniques used in this paper are also related to recent work by the first and second author in~\cite{MR3909634}, where they show hardness of approximation for some Morse matching problems.

In \Cref{sec:membership}, we show that \paraexptwo and \paraexp are both in \WP, and hence also \WP-complete, see \Cref{thm:inWP,thm:hardness,thm:inWPordered}. Both results rest on the key observation that a $2$-complex is erasable if and only if greedily collapsing triangles yields a $1$-dimensional complex (\Cref{prop:tancer}), as shown by Tancer~\cite[Proposition 5]{MR3439259}. 

In \Cref{sec:NPhardness} we show that, as a consequence of the above reduction, the problem of deciding whether a $2$-complex can be shown to be simple-homotopy equivalent to a $1$-complex using only $3$-dimensional expansions, called \dshe, is \NP-complete, see \Cref{thm:dshe}.

\section{Definitions and Preliminaries}

\subsection{Simplicial complexes}
A (finite) \emph{abstract simplicial complex} is a collection $\complex$ of 
subsets of a finite ground set $V$ such that if $\face$ is an element of $\complex$, and $\smallface$ is a
nonempty subset of $\face$, then $\smallface$ is an element of $\complex$. 
The ground set $V$ is called the \emph{set of vertices} of $\complex$.
Since simplicial complexes are determined by their facets, we sometimes present simplicial complexes by listing their facets.
A \emph{subcomplex} of $\complex$ is a subset $\altcomplex \subseteq \complex$ which is itself a simplicial complex.
Given a subset $W \subseteq V$ of the vertices of $K$, the \emph{induced subcomplex on $W$} 
consists of all simplices of $K$ that are subsets of $W$.

The elements of $\complex$ are referred to as its \emph{faces}. The \emph{dimension of a 
face} is defined to be its cardinality minus one, and the \emph{dimension of $\complex$} 
equals the largest dimension of its faces. For brevity, we sometimes refer to a $d$-dimensional
simplicial complex as a \emph{$d$-complex} and to a $d$-dimensional face as a \emph{$d$-face}. 
The $0$-, $1$-, and $2$-faces of a $d$-complex $\complex$ are called its \emph{vertices}, \emph{edges}, and \emph{triangles} respectively. 
Faces of $\complex$ which are not properly contained in any other face are called 
\emph{facets}.
An $(m-1)$-face $\smallface \in \complex$
which is contained in exactly one $m$-face $\face \in \complex$ is called 
\emph{free}. %

The \emph{star} of a vertex $v$ of complex $\complex$, written $\str_{\complex} (v)$, 
is the subcomplex consisting of all faces of $\complex$ containing $v$, together with 
their faces. If a map $\phi: V \to W $ between the vertex sets of two simplicial complexes 
$\complex$ and $\altcomplex$, respectively, sends every simplex $\sigma \in \complex$ to a simplex $\phi(\sigma) \in \altcomplex$, then the induced map $f: \complex \to \altcomplex, \sigma \mapsto \phi(\sigma)$, is said to be \emph{simplicial}.

\subsection{Simple-homotopy} \label{sec:simplehomotopy}

We introduce the basic notions of simple-homotopy used in the present paper. 
The general concept of simple-homotopy can be understood independently
from the notion of homotopy. In this sense this article aims to be self-contained.
For further reading on homotopy theory we refer to \cite{MR1867354}.

In short, a simple-homotopy equivalence is a refinement of a homotopy equivalence. 
It can be described purely combinatorially with the help of the following definition.  

\begin{definition}[Elementary collapses and expansions]
  Let $\complex_0$ be a simplicial complex, and let $\face, \smallface \in \complex_0$ 
  be an $m$-face and an $(m-1)$-face respectively such that $\smallface \subset \face$, and
  $\smallface$ is free in $\complex_0$. 

  We say that $\complex_1 = \complex_0 \setminus \{ \face , \smallface \}$
  arises from $\complex_0$ by an \emph{elementary collapse of dimension $m$} or \emph{elementary $m$-collapse}, denoted by 
  $\complex_0 \searrow \complex_1$. Its inverse, the operation 
  $\complex_0 = \complex_1 \cup \{ \face , \smallface \}$ is called an \emph{elementary expansion of dimension $m$} or \emph{elementary $m$-expansion}, written $\complex_0 \nearrow \complex_1$.
  If the complex is implicit from context, we denote elementary collapses by $\collapse{\smallface}{\face}$ and elementary expansions by $\expansion{\smallface}{\face}$.
  An elementary collapse or an elementary expansion is sometimes referred to as an \emph{elementary move}, or simply a \emph{move}.

  If there exists a sequence of elementary collapses turning a complex $\complex_0$ into $\complex_1$ we write $\complex_0 \searrow \complex_1$ and say that $\complex_0$ \emph{collapses to} $\complex_1$. If $\complex_1$ is one-dimensional, we say that $\complex_0$ is \emph{erasable}. If $\complex_1$ is merely a point we call $\complex_0$ \emph{collapsible}. 

Finally, we write $\complex_0 \nearrow \complex_1$ to indicate 
a sequence of expansions and say that $\complex_0$ \emph{expands} to $\complex_1$. 
\end{definition}

It follows that an expansion $\expansion{\smallface}{\face}$ can only be performed in a simplicial complex $\complex$ if all codimension $1$ faces of $\face$ except for $\smallface$ are already in $\complex$. Hence, let $\face$ be an $m$-face of a simplicial complex $\complex$, and let $\smallface$ be one of its $(m-1)$-faces. An \emph{($m$-dimensional) horn} $\horn (\face, \smallface)$ associated to the pair $(\face, \smallface)$ is the simplicial complex generated by the $(m-1)$-faces of $\face$ apart from $\smallface$. 

All $m$-expansions and $m$-collapses with $m>1$ leave the vertex set unchanged.

\begin{definition}[Simple-homotopy equivalence, simple-homotopy graph]
  Two simplicial complexes $\complex$ and $\altcomplex$ are said to be \emph{simple
homotopy equivalent} or of coinciding \emph{simple-homotopy type}, written $\complex \se \altcomplex$, 
if there exists a sequence $\SCC$ of elementary moves 
turning one into the other.
In this case, we write $\SCC: \complex \se \altcomplex$.

  The \emph{dimension of a simple-homotopy equivalence} is the maximum of the dimensions of $\complex$, $\altcomplex$ and 
of any elementary expansion or collapse in the sequence.

  The graph whose nodes are simplicial complexes, and two nodes are adjacent if their corresponding complexes are related by an elementary collapse is called \emph{simple-homotopy graph}. Naturally, its connected components are in one-to-one correspondence with simple-homotopy types.
\end{definition}

Two simplicial complexes of the same simple-homotopy type are homotopy equivalent, but the converse is not true, see, for instance, \cite{MR0005352}. For simple-homotopy equivalent simplicial complexes we know the following.

\begin{theorem}[Wall \cite{MR3439259}, Matveev \protect{\cite[Theorem 1.3.5]{MR2341532}}]
 \label{thm:deformation}
 Let $\complex$ and $\altcomplex$ be two simplicial complexes of the same simple-homotopy type and of dimension at most $m > 2$.
 Then there exists a simple-homotopy equivalence of dimension at most $m+1$, taking one to the other.
\end{theorem}

For the case $m=2$, \Cref{thm:deformation} is still open and known as the (topological or geometric) Andrews--Curtis conjecture~\cite{MR3024764,MR2341532,MR1921712}. On the other hand, it is known that any contractible $2$-complex is also simple-homotopy equivalent to a point \cite{MR0035437}. Hence, any pair of contractible $2$-complexes can be connected by a simple-homotopy equivalence of dimension at most four -- but determining whether we can always decide if such a simple-homotopy equivalence exists is an open question~\cite{MR1921705}, equivalent to the triviality problem for balanced group presentations~\cite{MR1279174}.

\subsection{Parametrized complexity}

Parametrized complexity, as introduced by Downey and Fellows in \cite{MR1656112}, is a refinement of classical complexity theory. The theory revolves around the general idea of developing complexity bounds for instances of a problem not just based on their size, but also involving an additional \emph{parameter}, which might be significantly smaller than the size.
Specifically, we have the following definition.

\begin{definition}[Parameter, parametrized problem%
]
  Let $\Sigma$ be a finite alphabet.
  \begin{enumerate}
    \item A \emph{parameter} of $\alphastrings$, the set of strings over $\Sigma$, is a function $\rho: \alphastrings \to \bbn$, attaching to every input $w \in \alphastrings$ a natural number $\rho(w)$.
    \item A \emph{parametrized problem} over $\Sigma$ is a pair $(P,\rho)$ consisting of a set $P\subseteq \alphastrings$ and a parametrization $\rho : \alphastrings \to \bbn$.
  \end{enumerate}
\end{definition}

In this article we consider the complexity class \WP for parametrized problems, following the definition by Flum and Grohe~\cite{MR2238686}.

\begin{definition}[Complexity Class \WP] 
 Let $\Sigma$ be an alphabet and $\rho: \alphastrings \to \bbn $  a parametrization.
A nondeterministic Turing machine $\bbm$ with input alphabet $\Sigma$ is
called $\rho$-\emph{restricted} if there are computable functions $f,h: \bbn\to\bbn$ and
a polynomial $p$ (with coefficients in the set of natural numbers) such that on every run with input $x\in \Sigma^{\ast}$ the
machine $\bbm$ performs at most $f(k)\cdot p(|x|)$ steps, at most $h(k)\cdot \log |x|$ of them
being nondeterministic, where $k := \rho(x)$.
\WP is the class of all parametrized problems $(P,\rho)$ that can be decided by a $\rho$-restricted nondeterministic Turing machine.
\end{definition}

\section{Problems}

In this article we consider the following parametrized problems. 

\begin{problem}[\paraexptwo]
{A $2$-dimensional simplicial complex $\complex$ and a natural number $p$.}
{$p$.}
{Is there a path in the simple-homotopy graph connecting $\complex$ to a $1$-complex using at most $p$ expansions?}
\end{problem}

\begin{problem}[\paraexp]
{A $2$-dimensional simplicial complex $\complex$ and a natural number $p$.}
{$p$.}
{Is there a path in the simple-homotopy graph connecting $\complex$ to a $1$-complex using first at most $p$ expansions, followed by a sequence of only collapses?}
\end{problem}

In \cref{sec:pComplexity}, we establish \WP-completenes for \paraexptwo and \paraexp.

The hardness proof works via a parametrized reduction using the \pMAS problem, which is a classical \NP-complete problem~\cite[p.~263]{MR519066} that is well-known to be \WP-complete with respect to the appropriate parameter~\cite[p. 473]{MR1656112}. 

\begin{problem}[\pMAS]
{A finite set $S$ of \emph{sentences}, an \emph{implication relation} $R$ consisting of pairs $(U,s)$ where $U \subseteq S$ and $s \in S$, and a positive integer $p\leq |S|$.}
{$p$.}
{Is there a set $S_0 \subseteq S$, called an \emph{axiom set}, with $|S_0| \leq p$ and a positive integer $n$ such that if we recursively define \[S_i := S_{i-1} \cup \{ s \in S \mid \exists U \subseteq S_{i-1} :(U, s) \in R\} \]
for $1 \leq i \leq n$, then $S_n = S$?}
\end{problem}

\begin{remark}
  \label{rem:restriction}
  Note that every instance of \pMAS can be reduced in polynomial time to an instance for which every sentence must occur in at least one implication relation: First iteratively remove all sentences from the instance which do not feature in at least one implication relation. Then, for each of them, reduce $p$ by one (note that each of them must necessarily be an axiom). It follows that solving the reduced instance is equivalent to solving the original instance.

  Similarly, note that if there exists an implication $(U,s) \in R$, $s \in U$, we can simply omit it and, if this deletes $s$ from the instance altogether, decrease $p$ by one.
\end{remark}

In  \cref{sec:NPhardness}, we show that the following variants of the expansion height problem are \NP-complete.

\begin{problemwp}[\dshe]
{A finite $2$-dimensional simplicial complex $\complex$ and a natural number $p$.}
{Is there a path in the simple-homotopy graph connecting $\complex$ to a $1$-complex using at most $p$ expansions, all of which are $3$-expansions? 
}
\end{problemwp}

\begin{problemwp}[\dshetwo]
{A finite $2$-dimensional simplicial complex $\complex$ and a natural number $p$.}
{Is there a path in the simple-homotopy graph connecting $\complex$ to a $1$-complex using first at most $p$ expansions, all of which are $3$-expansions, followed by a sequence of only collapses? 
}
\end{problemwp}

\section{Contractibility and collapsibility for $2$-complexes}

The main gadget used in the proof of our main result, \Cref{thm:hardness}, is based on the simplest $2$-dimensional contractible complex which is not collapsible to a point -- the \emph{dunce hat}. Hence, before we describe our main gadget in detail, we start this section by briefly discussing minimal triangulations of the dunce hat, and a variant that is collapsible through a unique free edge, the \emph{modified dunce hat}.

\subsection{The dunce hat}
\label{ssec:DH}

In the category of CW complexes, the dunce hat can be obtained by identifying two boundary edges of a triangle to build a cone and then gluing the third edge along the seam of the first gluing.
The resulting complex does not have a collapsible triangulation. On the other hand, the dunce hat is known to be contractible~\cite{MR0156351}.

The smallest simplicial complexes realizing this construction have $8$ vertices, $24$ edges and $17$ triangles. There are seven such minimal triangulations of the dunce hat~\cite{Paixao143SphereRec}. One such triangulation, denoted by $\dunce$, is shown in \Cref{fig:DH}.
The dunce hat $\dunce$ has two horns, namely $\horn (\{ 2,7,8 \} , \{ 1,2,7,8 \})$ and $\horn (\{ 3,5,6 \}, \{ 1,3,5,6 \})$, and hence admits two $3$-expansions, namely $\expansion{\{ 2,7,8 \}}{\{ 1,2,7,8 \}}$ and $\expansion{\{ 3,5,6 \}}{\{ 1,3,5,6 \}}$ respectively. They are shown by the shaded areas in \Cref{fig:DH}.

\begin{figure}[hb]
  \centering{
  \begin{tikzpicture}[scale=0.6]%
{ [every path/.style = {line cap=round,line join=round}]
    \coordinate (r) at (0,1);
    \coordinate (x) at ({sqrt(5/8 + sqrt(5)/8)},{-(1 - sqrt(5))/4});
    \coordinate (w) at ({sqrt(5/8 - sqrt(5)/8)},{-(1 + sqrt(5))/4});
    \coordinate (t) at ({-sqrt(5/8 - sqrt(5)/8)},{-(1 + sqrt(5))/4});
    \coordinate (v) at ({{-sqrt(5/8 + sqrt(5)/8)}},{-(1 - sqrt(5))/4});
    \coordinate (s) at (0,4);
    \coordinate (u) at ({-2*sqrt(3)},-2);
    \coordinate (p) at ({-2/3*sqrt(3)},-2);
    \coordinate (q) at ({2/3*sqrt(3)},-2);
    \coordinate (ur) at ({2*sqrt(3)},-2);
    \coordinate (ql) at ({-4/3*sqrt(3)},0);
    \coordinate (qr) at ({4/3*sqrt(3)},0);
    \coordinate (pl) at ({-2/3*sqrt(3)},2);
    \coordinate (pr) at ({2/3*sqrt(3)},2);
    
    \fill[verylightgray] (u)--(s)--(ur)--cycle;
    \fill[lightergray] 
    (u)--(p)--(t)--(v)--(ql)--cycle
    (ur)--(q)--(w)--(x)--(qr)--cycle
    ;

    \draw (u)--(ur);
    \draw[lightergray] (u)--(s)--(ur);
    \draw (w)--(x);
    \draw (x)--(r);
    \draw (r)--(v);
    \draw (u)--(v);
    \draw (t)--(v);
    \draw (t)--(w);
    \draw (t)--(p);
    \draw (w)--(p);
    \draw (w)--(q);
    \draw (w)--(ur);
    \draw (x)--(ur);
    \draw (x)--(qr);
    \draw (x)--(pr);
    \draw (r)--(pr);
    \draw (s)--(r);
    \draw (pl)--(r);
    \draw (pl)--(v);
    \draw (ql)--(v);
    \draw (u)--(t);
    \draw (r)--(t);
    \draw (r)--(w);
    
    \node[left=3pt] at (r) {$4$};
    \node[above right] at (x) {$8$};
    \node[above right] at (w) {$7$};
    \node[above left] at (t) {$6$};
    \node[above left] at (v) {$5$};
    \node[above,lightgray] at (s) {$1$};
    \node[left,lightgray] at (u) {$1$};
    \node[below] at (p) {$3$};
    \node[below] at (q) {$2$};
    \node[right] at (ur) {$1$};
    \node[left,lightgray] at (ql) {$3$};
    \node[right,lightgray] at (qr) {$2$};
    \node[left,lightgray] at (pl) {$2$};
    \node[right,lightgray] at (pr) {$3$};
}
\end{tikzpicture}
\hfil
  \begin{tikzpicture}[scale=.9]%
{ [every path/.style = {line cap=round,line join=round}]
    \coordinate (r) at (0,-1);
    \coordinate (w) at (-1,0);
    \coordinate (t) at (0,1);
    \coordinate (v) at (1,0);
    \coordinate (s) at (2,-2);
    \coordinate (u) at (2,2);
    \coordinate (sb) at (-2,2);
    \coordinate (q) at (-2,0);
    \coordinate (ur) at (-2,-2);
    \coordinate (ql) at (2,0);
    \coordinate (qr) at (0,-2);
     
    \fill[verylightgray] (u)--(s)--(ur)--(sb)--cycle;
    \draw[lightergray] (u)--(s)--(ur)--(sb)--cycle;
    \draw (ur)--(sb);
    \draw (r)--(v);
    \draw (sb)--(u);
    \draw (v)--(u);
    \draw (t)--(sb);
    \draw (t)--(v);
    \draw (t)--(w);
    \draw (w)--(q);
    \draw (w)--(ur);
    \draw (w)--(sb);
    \draw (r)--(ur);
    \draw (r)--(qr);
    \draw (s)--(r);
    \draw (s)--(v);
    \draw (ql)--(v);
    \draw (u)--(t);
    \draw (r)--(t);
    \draw (r)--(w);
    
    \node[left=1.5mm] at (r) {$6$};
    \node[right=1mm] at (w) {$5$};
    \node[above] at (t) {$4$};
    \node[left=1mm] at (v) {$7$};
    \node[below right,lightgray] at (s) {$3$};
    \node[above right] at (u) {$1$};
    \node[above left] at (sb) {$3$};
    \node[left] at (q) {$2$};
    \node[below left] at (ur) {$1$};
    \node[right,lightgray] at (ql) {$2$};
    \node[below,lightgray] at (qr) {$2$};

}
\end{tikzpicture}
  }
  \caption{Left: The $8$-vertex triangulation $\dunce$ of the dunce hat.  The two expansions turning it collapsible are highlighted.
  Right: The $7$-vertex triangulation $\gadget$ of the modified dunce hat. \label{fig:DH}\label{fig:punctDH}}
\end{figure}
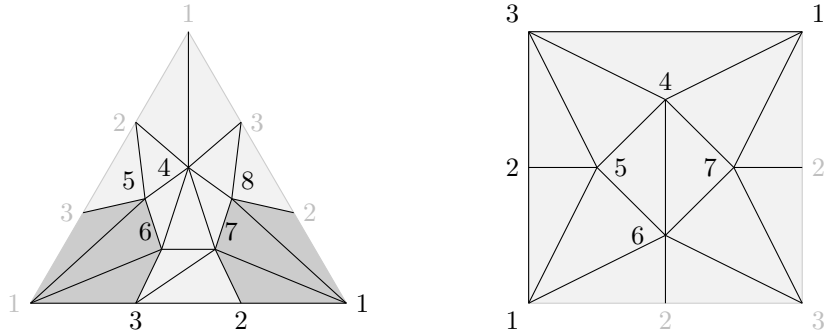

Note that after any of these two expansions we obtain a collapsible complex: After the expansion $\expansion{\{ 2,7,8 \}}{\{ 1,2,7,8 \}}$ and the collapses $\collapse{\{ 1,7,8 \}}{\{ 1,2,7,8 \}}$, $\collapse{\{ 1,7 \}}{\{ 1,2,7 \}}$, and $\collapse{\{ 1,8 \}}{\{ 1,2,8 \}}$, the edge $\{ 1,2\}$ becomes free and thus $\dunce$ becomes collapsible.
Similarly, starting with $\expansion{\{ 3,5,6 \}}{\{ 1,3,5,6 \}}$, one may perform the collapse $\collapse{\{ 1,5,6 \}}{\{ 1,3,5,6 \}}$ and proceed in an analogous way.
In particular, this shows that the dunce hat has the simple-homotopy type of a point, and in fact can be made collapsible by using a single expansion.

\subsection{The modified dunce hat}
\label{ssec:mDH}

Rather than working with the dunce hat directly, we base the construction of our gadget for the proof of \Cref{thm:hardness} on the \emph{modified dunce hat} \cite{Hachimori2000Combinatorics}. More precisely, we ``insert'' a free edge into the dunce hat. 
For instance,~\Cref{fig:punctDH} depicts a triangulation of the modified dunce hat, which we denote by $\gadget$, with $\{ 1,3 \}$ as the unique free edge. This particular triangulation of the modified dunce hat uses only $7$ vertices, $19$ edges, and $13$ triangles. 
The modified dunce hat has previously been used as a gadget to show hardness of approximation for Morse matchings \cite{MR3909634}.

If we assume that $\gadget$ is part of a larger complex $\complex$, in which edge $\{ 1,3 \}$ is glued to triangles not lying in $\gadget$, then  $\{ 1,3 \}$ is not free. In this case, the triangles of $\gadget$ can be collapsed away in essentially two distinct ways. Either, at some point in a simple-homotopy on $\complex$, the edge $\{ 1,3 \}$ becomes free and thus the triangles of $\gadget$ collapse, or the triangles of $\gadget$ become collapsible by performing one of two possible $3$-expansions on $\gadget$. Looking at the latter case in more detail, we have the following sequences of expansions and collapses:
$$
\expansion{\{ 2,5,6 \}}{\{ 1,2,5,6 \}}, \quad  \collapse{\{ 1,5,6 \}}{\{ 1,2,5,6 \}}, \quad \collapse{\{ 1,5 \}}{\{ 1,2,5 \}}, \quad \collapse{\{ 1,6 \}}{\{ 1,2,6 \}}
$$
and
$$
\expansion{\{ 2,6,7 \}}{\{ 2,3,6,7 \}}, \quad \collapse{\{ 3,6,7 \} }{\{ 2,3,6,7 \}}, \quad \collapse{\{ 3,6 \}}{\{ 2,3,6 \}}, \quad \collapse{\{ 3,7 \}}{\{ 2,3,7 \}}.
$$
In the first case, the edge $ \{ 1,2 \} $ is freed, in the second case, the edge $ \{ 2,3 \} $ is freed. Both sequences can be extended to a collapsing sequence of the entire complex $\gadget$.

\subsection{The main gadget}
\label{ssec:gadget}

Our gadget for the proof of \Cref{thm:hardness} is a subdivided version of the modified dunce hat $\gadget$ from \Cref{ssec:mDH}. More precisely, it is determined by two positive integers $m$ and $\ell$, denoted by $\gadget_{m,\ell}$, and can be constructed from the complex $\gadget$ in essentially two steps.

\begin{enumerate}
  \item Subdivide the edge $\{ 1,3 \}$ of $\gadget$ $(m-1)$ times, thereby introducing vertices $x_1, \ldots x_{m-1}$. Relabel $1 \to x_0$ and $3 \to x_m$ to obtain $m$ free edges $f_i = \{ x_{i-1} , x_{i} \}$, $1 \leq i \leq m$.
  \item Remove the edge $\{ 4,6 \}$ and place $\ell$ vertex-disjoint copies of the disk 
  $$ \{ \{ c_j , a_j, y_j \} , \{ c_j, y_j, z_j \}, \{ c_j, z_j, b_j \} , \{ d_j, a_j, y_j  \} , \{ d_j, y_j, z_j \}, \{ d_j, z_j, b_j \} \}, $$
  $1 \leq j \leq \ell$, inside the $4$-gon in the center of $\gadget$ bounded by $4$, $5$, $6$, and $7$. Triangulate the remaining space in the interior of the $4$-gon. This creates edges $e_j = \{ y_j, z_j \}$, $ 1 \leq j \leq \ell$, with pairwise vertex disjoint stars disjoint to $4$, $5$, $6$, and $7$ (now $\{ 4,6 \}$ reappears as a path from $4$ to $6$, and thus $\gadget_{m,\ell}$ is in fact a proper subdivision of $\gadget$). See \Cref{fig:gadget} for an illustration.
\end{enumerate}

\begin{figure}[hbt]
    \centering{
    \scriptsize
  \begin{tikzpicture}[scale=1.75,baseline]%
{ [every path/.style = {line cap=round,line join=round}]
    \coordinate (r) at (0,-1.6);
    \coordinate (w) at (-1.6,0);
    \coordinate (t) at (0,1.6);
    \coordinate (v) at (1.6,0);
    \coordinate (s) at (2,-2);
    \coordinate (u) at (2,2);
    \coordinate (sb) at (-2,2);
    \coordinate (q) at (-2,0);
    \coordinate (ur) at (-2,-2);
    \coordinate (ql) at (2,0);
    \coordinate (qr) at (0,-2);
    \coordinate (x1) at (-1.2,2);
    \coordinate (x2) at (-0.4,2);
    \coordinate (xm1) at (1.2,2);

    \coordinate (c0) at (-1.3,0);
    \coordinate (d0) at (-0.5,0);
    \coordinate (a0) at (-0.8,0.4);
    \coordinate (b0) at (-0.8,-0.4);
    \coordinate (y0) at (-0.95,0.15);
    \coordinate (z0) at (-0.95,-0.15);
    
    \coordinate (cl) at (0.5,0);
    \coordinate (dl) at (1.3,0);
    \coordinate (al) at (0.8,0.4);
    \coordinate (bl) at (0.8,-0.4);
    \coordinate (yl) at (0.95,0.15);
    \coordinate (zl) at (0.95,-0.15);

    \fill[verylightgray] (t)--(w)--(r)--(v)--cycle;
    \fill[lightergray] 
    (u)--(s)--(ur)--(sb)--cycle 
    (t)--(w)--(r)--(v)--cycle
    (a0)--(d0)--(b0)--(c0)--cycle
    (al)--(dl)--(bl)--(cl)--cycle
    ;
    
    \draw[lightergray] (u)--(s)--(ur)--(sb);

    \draw (sb)--(x2);
    \draw[dotted] (x2)--(xm1);
    \draw (xm1)--(u);

    \draw (ur)--(sb);
    \draw (r)--(v);
    \draw (v)--(u);
    \draw (t)--(sb);
    \draw (t)--(v);
    \draw (t)--(w);
    \draw (w)--(q);
    \draw (w)--(ur);
    \draw (w)--(sb);
    \draw (r)--(ur);
    \draw (r)--(qr);
    \draw (s)--(r);
    \draw (s)--(v);
    \draw (ql)--(v);
    \draw (u)--(t);

    \draw (r)--(w);

    \draw (t)--(x1);
    \draw (t)--(x2);
    \draw (t)--(xm1);
    
    \draw (w)--(c0);

    \draw (a0)--(d0)--(b0)--(c0)--cycle;
    \draw (y0)--(d0)--(z0)--(c0)--cycle;
    \draw (t)--(a0)--(y0)--(z0)--(b0)--(r);
    \draw (t)--(c0)--(r);
    \draw (t)--(d0)--(r);
    
    \draw[dotted] (d0)--(cl);
    
    \draw (al)--(dl)--(bl)--(cl)--cycle;
    \draw (yl)--(dl)--(zl)--(cl)--cycle;
    \draw (t)--(al)--(yl)--(zl)--(bl)--(r);
    \draw (t)--(cl)--(r);
    \draw (t)--(dl)--(r);

    \draw (dl)--(v);

    \node[below left] at (r) {$6$};
    \node[below left] at (w) {$5$};
    \node[above] at (t) {$4$};
    \node[below right] at (v) {$7$};
    \node[below right,lightgray] at (s) {$3$};
    \node[above] at (u) {$x_m = 1$};
    \node[above] at (sb) {$3 = x_0$};
    \node[left] at (q) {$2$};
    \node[below left] at (ur) {$1$};
    \node[right,lightgray] at (ql) {$2$};
    \node[below,lightgray] at (qr) {$2$};

    \node[above] at (x1) {$x_1$};
    \node[above] at (x2) {$x_2$};
    \node[above] at (xm1) {$x_{m-1}$};

    \node[above left=0mm and -1mm] at (c0) {$c_1$};
    \node[above right] at (d0) {$d_1$};
    \node[right] at (a0) {$a_1$};
    \node[right] at (b0) {$b_1$};
    \node[above right=-1mm and .5mm] at (y0) {\tiny $y_1$};
    \node[below right=-1mm and .5mm] at (z0) {\tiny $z_1$};    
    \node[right] at($(y0)!1/2!(z0)$) {$e_1$};

    \node[above left] at (cl) {$c_l$};
    \node[above right=0mm and -1mm] at (dl) {$d_l$};
    \node[left] at (al) {$a_l$};
    \node[left] at (bl) {$b_l$};
    \node[above left=-1mm and .5mm] at (yl) {\tiny $y_l$};
    \node[below left=-1mm and .5mm] at (zl) {\tiny $z_l$};    
    \node[left] at($(yl)!1/2!(zl)$) {$e_l$};

    \node[above] at($(sb)!1/2!(x1)$) {$f_1$};
    \node[above] at($(x1)!1/2!(x2)$) {$f_2$};
    \node[above] at($(xm1)!1/2!(u)$) {$f_m$};

}
\end{tikzpicture}
  }
  \caption{The main gadget of the construction $\gadget_{m,\ell}$. \label{fig:gadget}}
\end{figure}
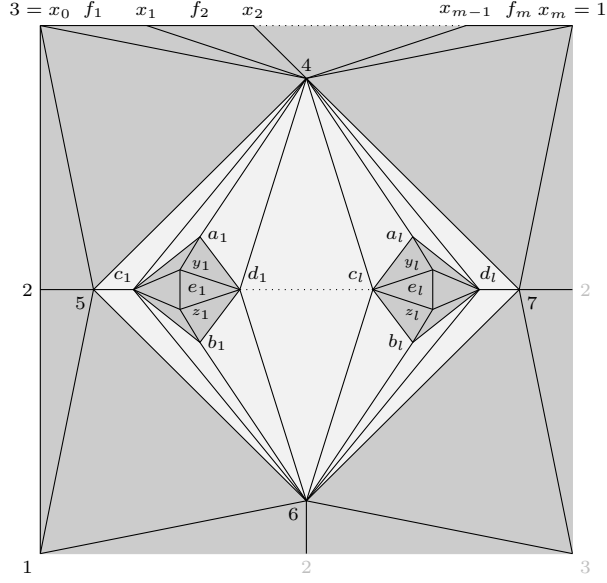

One key property of $\gadget_{m,\ell}$ is that we do not subdivide any faces of $\gadget$ near to the two available $3$-expansions. As a result, again, assuming that $\gadget_{m,\ell}$ is part of a larger complex $\complex$ where all free edges of $\gadget_{m,\ell}$ are glued to other triangles of $\complex$ outside of $\gadget_{m,\ell}$ and thus are not free, the triangles of $\gadget_{m,\ell}$ can be collapsed according to the following observation:

\begin{remark} 
Let $\complex$ be a two-dimensional simplicial complex such that $\gadget_{m,\ell}$ is a subcomplex whose vertices do not span any other faces of $\complex$ (i.e.,  $\gadget_{m,\ell}$ is an induced subcomplex of $\complex$), and $\complex \se \altcomplex $ where $\altcomplex$ is a $1$-complex.
Then, at least one of the following three statements holds true at some point in $\complex \se \altcomplex $, enabling us to eventually collapse away all the triangles of  $\gadget_{m,\ell}$.
\begin{enumerate}
\item one of the edges $f_i \in \gadget_{m,\ell}$ becomes free;
\item one of two $3$-expansions on $\gadget_{m,\ell}$ \emph{:} $\expansion{\{ 2,5,6 \}}{\{ 1,2,5,6 \}}$ or $\expansion{\{ 2,6,7 \}}{\{ 2,3,6,7 \}}$ is performed;
\item multiple expansions result in a complex in which all the triangles of $\gadget_{m,\ell}$ can be collapsed.
\end{enumerate} 
\end{remark}
In other words, if one of the edges $f_i \in \gadget_{m,\ell}$ does not become free at some point in $\complex \se \altcomplex $, then one is forced to use $3$-expansions (either directly on $\gadget_{m,\ell}$, or after performing additional expansions) to collapse away the triangles of $\gadget_{m,\ell}$. 

In \Cref{sec:hardness} we use this gadget to reduce an instance $A=(S,R,p)$  of \pMAS to \paraexptwo: Every sentence $s \in S$ is associated with one copy of $\gadget_{m,\ell}$, the edges $f_i$ correspond to implications $(U,s) \in R$, and the edges $e_j$ correspond to whenever $s \in U$ for some implication $(U,u) \in R$.

\section{Parametrized complexity of \paraexptwo}
\label{sec:pComplexity}

In this section, we first prove that \paraexptwo and \paraexp are \WP-hard by a reduction from \pMAS, a problem known to be \WP-complete. We then show that the two problems are also contained in \WP.

\subsection{\WP-hardness of expansion height problems}
\label{sec:hardness}

\begin{theorem}
 \label{thm:hardness}
\paraexptwo and \paraexp are \WP-hard problems.
\end{theorem}
The following lemma is used to assemble the gadgets in our reduction into a simplicial complex $\complex$.

\begin{lemma}[Munkres, \protect{\cite[Lemma 3.2]{MR755006}}]
\label{lem:pasting}
Let $C$ be a finite set, let $\complex$ be a simplicial complex with set of vertices $V$, and let $f: V \to C$ be a surjective map associating to each vertex of $\complex$ a color from $C$. The coloring $f$ extends to a simplicial map $g: \complex \to  \complex_f$ where $\complex_f$ has vertex set $C$ and is obtained from $\complex$ by identifying vertices with equal color.
  
  If for all pairs $v,w \in V$, $f(v) = f(w)$ implies that their stars $\str_{\complex} (v)$ and $\str_{\complex} (w)$ are vertex disjoint, then, for all faces $\face,\smallface \in \complex$ we have that 
\begin{itemize}
\item
$\face$ and $g(\face)$ have the same dimension, and 
\item
$g(\face) = g(\smallface)$ implies that either $\face = \smallface$ or $\face$ and $\smallface$ are vertex disjoint in $\complex$. 
\end{itemize}
\end{lemma} 
\Cref{lem:pasting} provides a way of gluing faces of a simplicial complex by a simplicial quotient map obtained from vertex identifications, and tells us when this gluing does not create unwanted identifications.

\begin{proof}[Proof of \Cref{thm:hardness}]

  We want to reduce \pMAS to \paraexptwo.

  Fix an instance $A = (S,R,p)$ of \pMAS such that every sentence $s \in S$ is subject to at least one implication $(U,s) \in R$ and such that $(U,s)  \in R$ implies $s \not \in U$. By \Cref{rem:restriction}, this is not a restriction, since every instance of \pMAS can be reduced to such an instance in polynomial time. For every sentence $s \in S$, take a copy $\gadget_s$ of the gadget $\gadget_{m,\ell}$ to model $s$, where $\ell \geq 0$ is the number of implications $(U,u) \in R$ with $s \in U$ and $m \geq 1$ is the number of implications $(U,s) \in R$. Thus, for all values $\ell \geq 0$ and $m \geq 1$ the gadget $\gadget_{m,\ell}$ is a simplicial complex without any unintended identifications. Denote the free edges of $\gadget_s$ by $f_i^s = \{ x_{i-1}^s,x_{i}^s \}$, $1 \leq i \leq m$, and its edges of type $e_j$ by $e_j^s = \{ y_j^s,z_j^s \}$, $1 \leq j \leq \ell$. 

For a fixed $s \in S$, endow the set of implications $(U,s) \in R$ of $s$ with an arbitrary order $(U_1,s), \ldots , (U_m,s)$. Similarly, for a fixed $u \in S$, order the set of implications in $R$ containing $u$ arbitrarily as $(U^1,s^1), \ldots , (U^{\ell},s^{\ell})$.
Now for every $(U_i,s) \in R$ and every $u \in U_i = U^j$ (i.e., $s^j = s$), glue the edge $f_i^s = \{ x_{i-1}^s,x_{i}^s \}$ of the gadget $\gadget_s$ to the edge $e_j^u = \{ y_{j}^u , z_j^u \}$ of $\gadget_u$ by identifying $x_{i-1}^s$ with $y_j^u$ and $x_{i}^s$ with $z_j^u$. 

  Performing these identifications for all implications in $R$ yields a complex, which we denote by $\complex$.
Note that, fixing $s \in S$, and $0 \leq i \leq m$, the only vertices to which $x_i^s$ can possibly be identified to in $\complex$ are $y_j^u, z_j^u$ ($ 1 \leq j \leq \ell$).

  More precisely, using the orderings from above for vertex $x_i^s$, let $(U_{i},s)$ and $(U_{i+1},s)$ be the $i$-th and $(i+1)$-st implication of $s$ in $R$ (if $i \in \{ 0, m\}$ there is only one implication to consider) and denote their sentences by $u_{1}, \ldots , u_{r}$ and $u^{1}, \ldots , u^{t}$, where $r = | U_i |$ and $t = |U_{i+1}|$. Moreover, let $U_{i}$ (resp.\@ $U_{i+1}$) be the $j_{f}$-th (resp.\@ $j^{g}$-th) implication where the sentence $u_{f}$ (resp.\@ $u^{g}$) occurs, for $ 1 \leq f \leq r$ (resp.\@ $ 1 \leq g \leq t$). 
Then $x_i^s$ is identified with $z_{j_{f}}^{u_{f}}$ ($ 1 \leq f \leq r$) and $y_{j^{g}}^{u^{g}}$ ($ 1 \leq g \leq t$).

Now since every fixed edge of type $e_j^u$ is only identified with 
one edge of type $f_i^s$, those vertices are not identified with any other vertices of $\complex$. 
Since, by construction, the set of the vertex stars of $z_{j_{f}}^{u_{f}}$ ($ 1 \leq f \leq r$), $y_{j^{g}}^{u^{g}}$ ($ 1 \leq g \leq t$), and $x_i^s$ are pairwise vertex disjoint, we can apply \Cref{lem:pasting} to ensure that no unwanted identifications occur in building up $\complex$. 
In particular, every gadget $\gadget_s$ is a subcomplex of $\complex$ via the canonical isomorphism given by the gluing map.

We now show that the following three statements are equivalent for our complex $\complex$:
\begin{enumerate}[(a)]
  \item there exists a simple-homotopy equivalence turning $\complex$ into a $1$-dimensional complex using first at most $p$ expansions, followed by a sequence of only collapses,
  \item there exists a simple-homotopy equivalence $\complex \se \altcomplex$ turning $\complex$ into a $1$-dimensional complex $\altcomplex$ using at most $p$ expansions, 
  and
  \item there exists an axiom set $S_0 \subset S$ for $A = (S,R,p)$ using at most $p$ elements.
\end{enumerate}

We trivially have that (a)$ \implies$(b).

In order to show that (c)$ \implies$(a), assume that there exists an axiom set $S_0 \subset S$ of size $p$, and perform one $3$-expansion on each gadget $\gadget_u$ with $u\in S_0$. As described in \cref{ssec:gadget}, these expansions admit all triangles of these gadgets to collapse. This, in turn, frees all edges $f_i^s$ where $s \in S$ has an implication $(U,s) \in R$ with $U \subset S_0$. Consequently, all triangles of such gadgets $\gadget_s$ can be collapsed. Since $S_0$ is an axiom set, repeating this process eventually collapses away all tetrahedra and triangles, leaving a $1$-complex.

  In order to show that (b)$ \implies$(c), we start with a few definitions. For a complex $\complex'$ with $\complex \se \complex'$, we say that a gadget $\gadget_s \subseteq \complex$ is \emph{\open} with respect to a simple-homotopy sequence $\SCC:\complex \se \complex'$ if at some point in the simple-homotopy sequence one of the triangles of $\gadget_s$ is removed. Otherwise $\gadget_s$ is said to be \emph{\closed}. Note that even if all triangles of $\gadget_s$ are present in $\complex'$, $\gadget_s$ might still be \open. Being \open or \closed is a property of the sequence $\SCC:\complex \se \complex'$, not of the complex $\complex'$.    
  
   We build the axiom set $S_0 \subset S$ for $A = (S,R,p)$ in the following way: A sentence $s \in S$ is in $S_0$ if and only if a triangle in $\gadget_s$ is removed by a $3$-collapse of the given simple-homotopy sequence $\SCC:\complex \se \altcomplex$. 
   We first inductively prove a claim about 
   \[S^k = \{ s\in S \mid \gadget_s \text{ is touched by a $3$-collapse in the first $k$ moves of } \complex \se \altcomplex \}. \]

\begin{claim} \label{cl:concise}
For $s \in S$, if the gadget $\gadget_s$ is \open by the first $k$ elementary moves in $\SCC:\complex \se \altcomplex$, then $s$ is implied by sentences in $S^k$.
\end{claim}
\begin{claimproof}
  First note that all gadgets are \closed in $\complex$ and $S^0 = \emptyset$.

By induction hypothesis, if a gadget $\gadget_s$ is \open by one of the first $k-1$ moves in $\SCC:\complex \se \altcomplex$, then $s$ is implied by sentences in $S^{k-1}$. 

The induction claim is trivially true if $\gadget_s$ is \open in the first $k-1$ moves, or if $\gadget_s$ is \open by a $3$-collapse in the $k$-th move ($s \in S^{k} \setminus S^{k-1}$), causing the sentence~$s$ to be included in $S_0$.

So, suppose that this is not the case. That is, suppose that $\gadget_s$ is \closed in the length~$k-1$ prefix $\SCC':\complex \se \complex'$ of $\SCC:\complex \se \altcomplex$ and \open by a $2$-collapse in the $k$-th move.
This implies that $S^k = S^{k-1}$ and that $\gadget_s$ is a subcomplex of $\complex'$ and one of the edges $f_i^s$ must be free in $\complex'$. 
Now let $\gadget_{u_1}, \gadget_{u_2},\dots, \gadget_{u_q}$ be the set of other gadgets containing triangles glued to $f_i^s$ in the original complex~$\complex$ (that is, there is an implication $(\{u_1, u_2, \dots , u_q\},s) \in R$). Since none of these triangles are present in $\complex'$, all gadgets $\gadget_{u_1}, \gadget_{u_2},\dots, \gadget_{u_q}$ must be \open in $\SCC':\complex \se \complex'$. Thus, either they were touched by a $3$-collapse and their corresponding sentences are part of $S^{k-1}$, or they were touched by a $2$-collapse and, by the induction hypothesis, their corresponding sentences are implied by sentences in $S^k$. It follows that $s$ is implied by sentences in $S^k = S^{k-1}$, proving the claim.
  \end{claimproof}

By assumption, $\complex$ is simple homotopy equivalent to a 1-complex $\altcomplex$. That is, $\SCC':\complex \se \altcomplex$ eventually removes all triangles from $\complex$. Hence, every sentence $s\in S$ is touched as a result of a $2$-collapse or a $3$-collapse. Let $m$ be the number of elementary moves needed to reach $\altcomplex$ starting from $\complex$.
Then, by \Cref{cl:concise}, $S^m = S_0$ is the desired axiom set. Also, since a sentence $s$ is included in $S_0$ only if a triangle belonging to gadget $\gadget_s$ is removed as part of a $3$-collapse, and since a triangle belonging to gadget $\gadget_s$ does not belong to any other gadget $\gadget_u$ for $u \neq s$, $S_0$ cannot contain more elements than the number of $3$-collapses (and hence $3$-expansions).

Finally, we infer \WP-hardness of \paraexptwo and \paraexp from the above equivalence and the \WP-hardness of \pMAS~\cite[p. 473]{MR1656112}.
\end{proof}

\subsection{\WP-membership of %
\paraexptwo}
\label{sec:membership}

We now show that \paraexp and \paraexptwo are in \WP by describing suitable nondeterministic algorithms for deciding both problems.
We begin with a well-known fact about checking collapsibility of $2$-complexes. 

\begin{proposition}[Tancer~\cite{MR3439259}, Proposition 5]
\label{prop:tancer}
Let $\complex$ be a $2$-complex that collapses to a $1$-complex $\altcomplex$ and to another $2$-complex $M$. Then $M$ also collapses to a $1$-complex.
\end{proposition}

\begin{remark} \label{rem:greedy}
The proposition above implies that we can collapse an input $2$-complex $\complex$ greedily until no more $2$-collapses are possible, and if $\complex$ collapses to a $1$-complex $\altcomplex$, the algorithm is guaranteed to terminate with a $1$-complex as well.
\end{remark}

\begin{theorem}
 \label{thm:inWPordered}
 \paraexp  is in \WP.
\end{theorem}

\begin{proof}
Let $\complex$ be a simplicial complex with $n$ simplices. First, note that if there exists a simple homotopy sequence $\SCC$ taking $\complex$ to a $1$-complex with $p$ expansions that all come at the beginning of the sequence, then there also exists a simple homotopy sequence $\SCC_\bbm$ taking  $\complex$ to a $1$-complex where $p$ expansions are followed by collapses such that, for each~$d$, all collapses of dimension $d+1$ are executed before collapses of dimension $d$. 
This follows from observing that, for any two $d$-collapses $\collapse{\sigma}{\tau}$ and $\collapse{\alpha}{\beta}$, if the $d$-collapse $\collapse{\sigma}{\tau}$ is executed before the $d$-collapse $\collapse{\alpha}{\beta}$ in $\SCC$, then the same can be carried out in $\SCC_\bbm$.
Also, in any simple homotopy sequence $\SCC$ that takes $\complex$ to a $1$-complex, for every $d>2$, the number of $d$-expansions equals the number of $d$-collapses in $\SCC$. This follows from a simple inductive argument starting with highest dimensional moves. 

 Denoting the total number of $d$-collapses, $d>2$, in $\SCC$ by $q\leq p$, it follows that, if there exists a simple homotopy sequence $\SCC$ with $p$ expansions that come at the beginning, then there exists a simple homotopy sequence $\SCC_\bbm$ with $p$ expansions in the beginning followed by $q$ collapses that gives rise to a $2$-complex $\complex'$ with $O(n^3)$ faces. The faces can be as many as $O(n^3)$ since $\bbm$ does not guess any $2$-collapses. Furthermore, if $\KCC$ is erasable through the simple homotopy sequence $\SCC$, then $\complex'$ is also erasable, once again, because the $2$-collapses of $\SCC$ can be carried out in the same order in $\SCC_\bbm$. Hence, the non-deterministic Turing machine $\bbm$ can now be described as follows:
 \begin{enumerate}
 \item Guess $p$ expansions and $q$ collapses non-deterministically to obtain a complex $\complex'$.
 \item Deterministically check if $\complex'$ is erasable.
 \end{enumerate}
By~\cref{rem:greedy}, erasability of $\complex'$ can be deterministically checked in time polynomial in $n$. 

Since any simplex in the desired simple homotopy sequence has at most $n+p$ vertices, the number of bits required to encode a single vertex is $O( \log (n+p))$. Also, because the dimension of the faces involved in expansions and collapses is certainly in $O(p)$, and since an expansion or a collapse can be fully described by a pair of simplices, the number of bits required to encode an expansion or a collapse is $O(p\log (n+p))$. Hence, in order to guess $p+q$ moves, it suffices for $\bbm$ to guess $O(p^2\cdot\log (n+p))$ bits in total since $q\leq p$. Now, assuming $n,p \geq 2$, we have
\[
p^2 \log (n+p) \leq p^2 \log(np) = p^2 \log(n) + p^2 \log(p)  
		      \leq p^2 (1+\log(p)) \log(n).
\]
Hence, for sufficiently large $n$ and $p$, the number of bits guessed by  $\bbm$ is bounded by a function of the form $f(p)\log n$.
Thus, $\bbm$ is a $p$-restricted Turing machine, and \paraexp  is in \WP.
\end{proof}

\begin{theorem}
 \label{thm:inWP}
 \paraexptwo  is in \WP.
\end{theorem}

\begin{proof}
Assume that there exists a simple homotopy that takes $\complex$ to a $1$-complex using no more than $p$ expansions. The Turing machine $\bbm$ needs to generate one such sequence. Below, we show that, in order to achieve this, $\bbm$ does not have to guess an entire simple homotopy sequence, but only a subsequence, and the remaining part of the sequence can be found deterministically by $\bbm$.
  
  Given a $2$-dimensional complex $\complex$ with $n$ faces, $\bbm$ first nondeterministically guesses $p$ expansions and $p$ collapses, and the order in which they are to be executed. These moves are referred to in the following as \emph{prescribed moves}. While these moves are meant to appear in a specified order, they need not appear consecutively. The moves that are not prescribed are computed deterministically by $\bbm$. A simple homotopy sequence of $\complex$, that takes $\complex$ to a $1$-complex, in which all the prescribed moves occur as a subsequence, is called a sequence \emph{compatible} with the prescribed moves. By assumption, there exists a set of prescribed moves for which a compatible sequence exists.

In order to give a description of $\bbm$, we introduce some additional terminology. 
Let $\SCC^j_q$ be an intermediate simple homotopy sequence computed by $\bbm$, such that the first $j$ prescribed moves guessed by $\bbm$ form a subsequence of $\SCC^j_q$, $q$ is the total number of moves in $\SCC^j_q$, and $\SCC^j_q$ is a prefix of a set of sequences $\SSS$ compatible with the prescribed set of moves. Let $ \complex \se \complex^j_q$ be the complex obtained by executing the moves in $\SCC^j_q$. Then, a collapse $\collapse{\sigma}{\tau}$ in $K^j_q$ is \emph{valid} for this prefix if appending the collapse still leaves a compatible prefix. That is, $\SCC^j_q$ appended with the collapse $\collapse{\sigma}{\tau}$ (giving  $\SCC^j_{q+1}$) continues to be a prefix of at least one compatible sequence $\SCC \in \SSS$. A collapse that is not valid is said to be \emph{forbidden}.

Note that labelling vertices of a complex $C$ by natural numbers determines a lexicographic order on the simplices of $C$. The lexicographic order $<_C$ on simplices of $C$ can be extended to a lexicographic order $\prec$ on collapses as follows: 
If  $(\collapse{\sigma}{\tau}), (\collapse{\alpha}{\beta})$ are two collapses in $C$, then $(\collapse{\sigma}{\tau}) \prec (\collapse{\alpha}{\beta})$ if $\sigma <_C \alpha$.

  The Turing machine $\bbm$ for deciding \paraexptwo can  be described as follows: 
\begin{enumerate}
 \item  Guess $2p$ prescribed moves non-deterministically.
 \item Execute $2$-collapses in lexicographic order until no more $2$-collapses are valid.
 \item Repeat until all prescribed moves have been executed:
  \begin{enumerate}
  \item Execute the next prescribed move.
  \item Execute $2$-collapses in lexicographic order until no more collapses are valid.
 \end{enumerate}
\end{enumerate}
  Let $\SCC$ be a sequence compatible with an ordered set of prescribed moves $X$ (of cardinality $2p$), and let $\SCC_\bbm$ be a simple homotopy sequence computed by $\bbm$ as above such that $X$ is a subsequence of $\SCC_\bbm$.
Now, let $\sigma$ be a free edge associated with a $2$-collapse $\collapse{\sigma}{\tau}$ in $\SCC^j_q$ for some $j$ and $q$, where $\SCC^j_q$ is a subsequence of $\SCC_\bbm$.
If there exist future prescribed moves including cofaces of $\sigma$, then the next prescribed move including cofaces of $\sigma$ that is not an expansion involving $\sigma$ is denoted by $m_1$. Similarly, if there exist future prescribed moves including cofaces of $\tau$, then the next prescribed move including cofaces of $\tau$ that is not an expansion involving $\tau$ is denoted by $m_2$.
Note that $m_2$ cannot come before $m_1$ but we may have $m_1 = m_2$. Then, the $2$-collapse $\collapse{\sigma}{\tau}$ is forbidden if and only if $m_1$ exists and is not preceded by a future prescribed expansion involving $\sigma$ or $m_2$ exists, and is not preceded by a future prescribed expansion involving $\tau$. %
It follows that, for each free edge, checking if a collapse is forbidden (or valid) can be done deterministically in time polynomial in $p$. To see this note that the most expensive atomic operation is to check if a simplex (of dimension $1$ or $2$) is a face of a simplex that is at most $p$ dimensional, and  the number of prescribed moves is at most $2 p$. Altogether, the set of valid collapses can be computed in time polynomial in $n$ and $p$, which can also be lexicographically ordered in polynomial time.

  Finally, let $\complex'$ denote the complex obtained from $\complex$ by the sequence $\SCC_\bbm$. Then, the following claim establishes the effectiveness of the greedy strategy employed by $\bbm$.
\begin{claim}  
If there exists a simple homotopy sequence with at most $p$ expansions that takes $\complex$ to a $1$-complex, then there exists an execution branch of the Turing machine that terminates successfully, i.e., the complex $\complex'$ obtained by $\bbm$ is a $1$-complex. 
\end{claim}
\begin{proof} 
Let $\SCC$ be a simple homotopy sequence with $p$ expansions that takes $\complex$ to a $1$-complex.
Let $X_e$ be the ordered set of expansions in $\SCC$. Thus, $|X_e| = p$.
Moreover, let $X_e^+$ ($X_c^+$) denote the $d$-expansions ($d$-collapses) in $\SCC$ with $d > 2$. As in \Cref{thm:inWPordered}, by a simple inductive argument starting from the highest dimension it can be shown that $|X_c^+| = |X_e^+|$. To the $p$ expansions of $\SCC$, we associate a set $X_c$ of  collapses of $\SCC$ as follows: If $|X_c^+| < p$, then let $X_c^-$ be an arbitrary set of $d$-collapses in $\SCC$ with $d \leq 2$, and $|X_c^-| =  p  - |X_c^+| $. Now, let $X_c = X_c^+ \cup X_c^-$, so that $|X_c| = p$. 
Finally, let the ordered set $X$ of prescribed moves  be the set containing all elements of $X_c \cup X_e$ seen as a subsequence of $\SCC$.

We assume that the non-deterministic Turing machine $\bbm$ correctly guesses the specified sequence of prescribed moves $X$.
It now suffices to show the following claim about the sequence $\SCC_\bbm$.\end{proof}

\begin{claim} 
 $\SCC_\bbm$ is compatible with the prescribed moves $X$, and  $\SCC_\bbm$ takes $\complex$ to a $1$-complex if $\SCC$ takes $\complex$ to a $1$-complex.
\end{claim}
\begin{proof}
Let $\complex \se \complex^j$  denote the complex obtained from $\SCC$ after executing the $j$-th prescribed move in $\SCC$. 
We show that there exists a complex $\complex^j_\bbm$ obtained from $\SCC_\bbm$ after executing the $j$-th prescribed move in $\SCC_\bbm$. Also, let $T^j$ ($T^j_\bbm$ ) denote the set of $2$-simplices of $\complex^j$ ($\complex^j_\bbm$).  

First observe that, $\complex^0_\bbm = \complex$ exists and that $T^0_\bbm = T^0$.
We now show that the following claim is inductively true: $T^j_\bbm \subset T^j $ for all $j \in [1,2p]$. 
Suppose we make the induction hypothesis that $T^{j-1}_\bbm \subset T^{j-1} $ for some $j \in [1,2p]$.
Then, the set of forbidden collapses for $\SCC$ and $\SCC_\bbm$ are the same until the $j$-th move in $X$ can be reached. Let $\tau_1$ be the first  $2$-face of $K^{j-1}$ that is removed as part of a $2$-collapse after $j-1$ prescribed moves have been executed in $\SCC$. Without loss of generality, assume that the $2$-collapse that removes $\tau_1$ is non-prescribed. Then, there exists an edge $\sigma_1 \subset \tau_1$ such that $\tau_1$ is the unique coface of $\sigma_1$ in $K^{j-1}$. By induction hypothesis, since $T^{j-1}_\bbm \subset T^{j-1} $ the same is also true for $K^{j-1}_\bbm$. Since $\bbm$ greedily removes every valid collapse it can (in lexicographic order), at some appropriate lexicographic index, $\tau_1$ is also removed from $K^{j-1}_\bbm$ (possibly along with $\sigma_1$). Now, let $\tau_1, \tau_2, \dots, \tau_{q-1}$ be the first $q-1$ $2$-faces removed from $K^{j-1}$ (as part of non-prescribed collapses). Assume that $\tau_1, \tau_2, \dots, \tau_{q-1}$ have also been removed from $K^{j-1}_\bbm$. By the same reasoning as before, if $\tau_q$ is the $q$-th face to be removed from $K^{j-1}_\bbm$  (as part of non-prescribed collapses), then $\tau_q$ may also be removed from  $\SCC_\bbm$ as part of a valid collapse. Hence, by induction, $T^j_\bbm \subset T^j $ for all $j \in [1,2p]$. 

Finally, since by assumption, $\complex^{2p}$ collapses to a $1$-complex, by applying arguments analogous to the induction above, the same is true for $\complex^{2p}_\bbm$ since $T^{2p}_\bbm \subset  T^{2p} $.
\end{proof}

 Since given a $2$-complex with $n$ faces, $\bbm$ non-deterministically guesses $2p$ moves, as in~\Cref{thm:inWPordered}, the number of bits guessed by $\bbm$ is bounded by  $f(p)\log(n)$, where $f(p) = O(p^2 (1+\log(p)))$.
Hence, $\bbm$ is a $p$-restricted Turing machine, and \paraexptwo is in \WP.
\end{proof}

\section{\NP-completeness of \dshe}
\label{sec:NPhardness}

Note that the parametrized reduction from \pMAS to \paraexptwo (and \paraexp) is also a polynomial-time reduction (or Karp reduction) from \MAS to \dshe (and \dshetwo), since the complexity of reduction is independent of the parameter $p$ and depends only on the size of the input complex. This observation leads us to the following result.

\begin{theorem}
\label{thm:dshe}
The decision problems \dshe and \dshetwo are \NP-hard.
\end{theorem} 

\begin{proof}
Since the \MAS problem is known to be \NP-hard~\cite{MR519066}, it follows that \dshe and \dshetwo are also \NP-hard.
\end{proof}

For the rest of the section, we assume that $\complex$ is a $2$-complex $\complex$ with $n$ faces and $m$ vertices. The total number of simplices that one can encounter in any simple homotopy sequence of $\complex$ using only $3$-expansions is at most $M = O(m^4)$. (Note that the ground set of $\complex$ is fixed since we do not allow $1$-expansions). Hence, the total number of \emph{elementary moves} that may be available at any given point in the sequence is bounded by $O(M)$. That is, $p$ itself is bounded by $O(M)$.

\begin{theorem}
 \label{thm:inNPone}
 \dshe  is in \NP.
\end{theorem}
\begin{proof}
  The non-deterministic algorithm $\bbm$ for deciding \dshe first  guesses at each point in the simple homotopy sequence  starting with $\complex$, one elementary move (out of at most $O(M)$ available moves), and constructs a new complex from the move. The total number of moves made by $\bbm$ is bounded by $(\frac{n + 2 p -1}{2})$. Finally, $\bbm$ checks if the final complex is a $1$-complex. 
\end{proof}

\begin{theorem}
 \label{thm:inNPtwo}
 \dshetwo  is in \NP.
\end{theorem}
\begin{proof}
  The non-deterministic algorithm $\bbm$ for deciding \dshetwo first  guesses at most $p$ $3$-expansions followed by an equal number of $3$-collapses, resulting in a $2$-complex $\complex'$
  with $n$ faces. From~\cref{rem:greedy}, the erasability of $\complex'$ can be deterministically checked in time polynomial in $n$, proving the claim.
  \end{proof}

\bibliographystyle{plainurl}%
\bibliography{expansion}

\end{document}